\documentclass[a4paper, oneside, english, reqno]{amsart}

\usepackage[utf8]{inputenx}
\usepackage{amsthm, amsmath, amssymb, amsfonts, mathrsfs, mathtools, stmaryrd}
\usepackage{textcomp, url, enumerate, latexsym, graphicx, varioref, multirow, layout, siunitx, booktabs}
\usepackage{hyperref}
\usepackage{pdflscape}
\usepackage{verbatim}
\usepackage{geometry}
\usepackage{float}
\usepackage{babel}
\usepackage[table]{xcolor}
\usepackage{tikz, rotating, tikz-cd, pigpen}
\usetikzlibrary{matrix, arrows, decorations.pathmorphing}
\usetikzlibrary{tikzmark,calc}
\usetikzlibrary{positioning}

\RequirePackage[color       = blue!20,
                bordercolor = black,
                textsize    = tiny,
                figwidth    = 0.99\linewidth]{todonotes}

\usepackage[all]{xy}
\usepackage{microtype}
 \usepackage{csquotes}

\RequirePackage{enumitem}
\setlist[itemize]{font = \upshape, before = \leavevmode}
\setlist[enumerate]{font = \upshape, before = \leavevmode}
\setlist[description]{before = \leavevmode}

\usepackage[backend = biber, style = numeric, isbn = false, doi = false]{biblatex}
\AtBeginBibliography{\small}
\everymath
{
    \ifodd\value{page}          
        \allowdisplaybreaks[1]  
        \else                   
        \allowdisplaybreaks[4]  
    \fi                         
}

\RequirePackage{zref-clever}
\zcsetup{
    abbrev=true, 
    }

\newcommand{\cref}[1]{\zcref{#1}}
\newcommand{\Cref}[1]{\zcref[S]{#1}}

\ExplSyntaxOn
\newcommand{\zcrefglobalstringname}[2]{
  \__zrefclever_opt_varname_lang_type:nnnn{\languagename}{#1}{#2}{tl}
}
\newcommand{\zcreflocalstringname}[2]{
  \__zrefclever_opt_varname_type:een{#1}{#2}{tl}
}
\newcommand{\zcrefgetstring}[2]{
  \__zrefclever_provide_langfile:e { \languagename }
  \__zrefclever_process_language_settings:
  \ifcsvoid{\zcreflocalstringname{#1}{#2}}
    {\csname \zcrefglobalstringname{#1}{#2} \endcsname}
    {\csname \zcreflocalstringname{#1}{#2} \endcsname}
}
\ExplSyntaxOff

\NewDocumentCommand{\newzctheorem}{mO{#1}m}{
  \newtheorem{#1}[sharedtheoremcounter]{#3}
    \AddToHook{env/#1/begin}{%
      \zcsetup{countertype={sharedtheoremcounter=#1}}}
}

\newzctheorem{theorem}{Theorem}
\newzctheorem{lemma}[theorem]{Lemma}
\newzctheorem{scholium}[theorem]{Scholium}
\newzctheorem{prop}[theorem]{Proposition}
\newzctheorem{conjecture}[theorem]{Conjecture}
\newzctheorem{corollary}[theorem]{Corollary}
\newzctheorem{problem}[theorem]{Problem}
\newzctheorem{statement}[theorem]{Statement}

\theoremstyle{definition}
\newzctheorem{definition}[theorem]{Definition}
\newzctheorem{assumption}[theorem]{Assumption}
\newzctheorem{defprop}[theorem]{Definition/Proposition}
\newzctheorem{hypothesis}[theorem]{Hypothesis}
\newzctheorem{example}[theorem]{Example}
\newzctheorem{question}[theorem]{Question}
\newzctheorem{construction}[theorem]{Construction}
\newzctheorem{idea}[theorem]{Idea}

\theoremstyle{remark}
\newzctheorem{remark}[theorem]{Remark}
\newzctheorem{notation}[theorem]{Notation}
\newzctheorem{remarks}[theorem]{Remarks}

\DeclarePairedDelimiter{\p}{\lparen}{\rparen}          
\DeclarePairedDelimiter{\floor}{\lfloor}{\rfloor}      

\newcommand{\N}{\mathbb{N}}

\newcommand{\Z}{\mathbb{Z}}

\newcommand{\bfc}{\mathbf{c}}

\newcommand{\bfe}{\mathbf{e}}

\newcommand{\calC}{\mathcal{C}}

\newcommand{\id}{\mathrm{id}}

\newcommand{\defeq}{\vcentcolon=}

\begin{document}

\title{A symmetry approach to number tricks}
\date{}
\author{Håkon Kolderup}
\maketitle

\begin{abstract}
\noindent We generalize the classical ``1089-number trick", which states that a certain combination of addition, subtraction and swapping the digits of a three-digit number will always output 1089. More precisely, we show that any pair of zero divisors $f\circ g=0$ in the group ring $\Z[\Sigma_n]$ on the $n$-th symmetric group gives rise to a partition of the set of $n$-digit numbers into subsets $U_\bfe$ defined by linear inequalities, such that the zero divisors act constantly on each $U_\bfe$ and hence define a number trick.
\end{abstract}

\section{Introduction}
A well-known ``number trick" proceeds as follows. Take any three-digit number $abc$ with $a>c$ and subtract its reverse $cba$. Then, to this difference\footnote{The $1089$-trick outputs the result $1089$ for all three-digit numbers $abc$ with $a>c$, but we have to remember to treat also the difference as a three-digit number. For example, if $a=c+1$, then the difference $abc-cba$ is $099$ as a three-digit number. Thus, reversing again and adding we get $990+99=1089$.} add the reverse of the difference. The answer is always $1089$.

Taking for instance the number $593$, we first reverse it to get $395$. Upon subtracting the reverse we get $593-395=198$. Finally we add to this the reverse of our answer, giving $198+891=1089$.
Spelling this out in general with $abc=a\cdot 10^2+b\cdot10+c$ reveals that the coefficients $a$, $b$ and $c$ cancel out and we are left with what remains after carrying, which sums to 1089.\newline

The 1089-number trick has been featured in various media, books \cite{Acheson}, and research papers \cite{Yannis,Behrends,1089}. Almirantis and Li \cite{Yannis} iterated the steps of the 1089-trick and studied the resulting dynamical system, while Behrends \cite{Behrends} and Webster \cite{1089} considered a generalization of the 1089-trick to $n$-digit numbers by using the reverse of an $n$-digit number and then applying the recipe of the 1089-trick. The papers \cite{Behrends,1089} moreover relate the number of possible outputs of these generalized 1089-tricks to Fibonacci numbers.\newline

The aim of this note is to provide several new examples of such number tricks, and more generally to set up a formalism allowing the reader to easily discover their own tricks. A new example we will encounter is the \emph{rotation trick}, in which we start as in the $1089$-trick with a three-digit number $abc$, this time with $a\ge b>c$, and subtract the \emph{rotated number} $cab$. We write the difference $abc-cab$ as a three-digit number $def$ and finally add together all rotations of this number: $def+fde+efd$. The answer is always 1998.

What is the underlying mechanism behind the $1089$-trick, the generalized $1089$-tricks of \cite{Behrends,1089}, and the rotation trick? All of these deal in some way with a permutation of the digits of the input number, and after suitably subtracting and adding these permutations the digits end up canceling, yielding a constant end result. The reason why the digits cancel is that the underlying linear combination of permutations, considered as an element in the group ring $\Z[\Sigma_n]$, sums to zero. Based on this observation, we will in this note generalize the number tricks encountered above to a class of number tricks that stem from a null relation in $\Z[\Sigma_n]$. More precisely, we define an action of $\Z[\Sigma_n]$ on $n$-digit numbers, and prove that the action of any pair of zero divisors $f,g\in\Z[\Sigma_n]$ 
depends only on the carrying and not on the input numbers. We show furthermore that the carrying is locally constant, and hence that the relation $f\circ g=0$ defines a number trick.

\subsection*{Overview} \Cref{section:informal} is an informal discussion exemplifying the main points of this paper via the classical 1089-trick as well as new number tricks. \Cref{sect:mainthm} contains the main technical arguments and formalizes the discussion in \Cref{section:informal}. Finally, in \Cref{table} we revisit the examples of \Cref{section:informal} in more detail.

\subsection*{Notation} Below follows an overview of the notation we use:
\[\begin{array}{l|l}
\Sigma_n, \ \Z[\Sigma_n] & \text{Symmetric group on } n\text{ letters, integral group ring on }\Sigma_n\\
\N & \text{The set }\{0,1,2,\dots\}\text{ of natural numbers}\\
\N_n & \text{The set }\{a_{1}a_2\dots a_n\defeq \sum_{i=1}^{n}a_i10^{n-i}:0\le a_i\le 9\}\subseteq\N \text{ of }n\text{-digit numbers}\\
V_n & \text{The set }\{0,1,\dots,9\}^n\subseteq\Z^n \text{ of length-}n\text{ digit vectors} \\
g\cdot v,\ f\circ g & \text{Action of }g\in\Z[\Sigma_n]\text{ on }v\in \Z^n\text{, product (in }\Z[\Sigma_n]\text{) of }f,g\in\Z[\Sigma_n]\\
\Phi\colon \Z^n\to \Z & \text{The evaluation homomorphism }\Phi(x_1,\dots,x_n)=\sum_{i=1}^n x_i10^{\,n-i}\\
N\colon \Z^n\to V_n & \text{Normalization map, defined in \Cref{def:N}}
\end{array}
\]
We note that there is no real distinction between $\N_n$ and $V_n$: since we allow zero as leading coefficients for $n$-digit numbers, a vector $(a_1,\dots,a_n)\in V_n$ corresponds bijectively to the $n$-digit number $a_1 a_2\dots a_n\in\N_n$. Thus we may use the terms ``$n$-digit number" and ``length-$n$ digit vector" interchangeably.

\subsubsection*{Acknowledgments}
I am grateful to Karl Erik Holter for his interest which encouraged me to write this note. I would also like to thank the anonymous referee for many helpful comments and remarks.

\section{Zero divisors in the symmetric group ring give rise to number tricks}\label{section:informal}
We can think of the 1089-number trick as the computation of the action on a three-digit number $abc$ of the product of linear combinations of permutations
\[
(1+\tau)\circ (1-\tau)\in\Z[\Sigma_3],
\]
where $\tau=(13)\in\Sigma_3$ is the transposition $\tau(abc)=cba$. To see this, pick a three-digit number $abc$ with $a>c$ and write the difference $abc-cba$ as a three-digit number $def$. By computing the action of the operator $(1+\tau)\circ(1-\tau)$ on $abc$ we mean
\begin{align*}
(1+\tau)\circ (1-\tau)(abc)&=(1+\tau)(abc-cba)\\
&=(1+\tau)(def)\\
&=def+fed,
\end{align*}
which is precisely the $1089$-trick.

Now, as $\tau^2=1$, the element $(1+\tau)\circ (1-\tau)=1-\tau^2$ is zero in $\Z[\Sigma_3]$. This is the reason why the digits of $abc$ cancel.
Based on this observation, our aim is to provide a general setting for number tricks of this sort using any pair of zero divisors in $\Z[\Sigma_n]$.

\subsection{Formal and spectator friendly number tricks} This approach of using any pair of zero divisors from $\Z[\Sigma_n]$ is however quite general, and the associated number trick may not always be suitable to perform in front of a spectator. In fact, this happens already for the 1089-trick: if $a<c$, it is not immediately clear how to proceed since the initial subtraction $abc-cba$ takes us outside the set $\N_3$. There are several possible interpretations of the 1089-trick for $a<c$: for instance, one may allow negative numbers so that $-1089$ will be the result of the trick. Another possibility, relating to the approach of \cite{Yannis}, is to declare that the smaller number should always be subtracted from the larger number. In order to obtain a streamlined approach to general number tricks we will in this note take a third approach, which simultaneously encodes the carrying that occurs in the $1089$-trick, and also solves the problem of potentially landing outside $\N_n$. More precisely, we will in \Cref{def:N} construct a \emph{normalization map} $N\colon\Z^n\to V_n$, which takes any $n$-tuple of integers and produces an $n$-tuple of integers between $0$ and $9$. This normalization process is an algorithm generalizing the usual carry operation, working from right to left producing carries $c_i$ at each stage (see \Cref{def:N} for details).

Our interpretation of a number trick is then as follows. Start with a pair of zero divisors $f\circ g=0\in\Z[\Sigma_n]$; compute the action $g\cdot v$ of $g$ on  $v\in V_n$; normalize according to \Cref{def:N}; let $f$ act on this normalization and finally compute the result $\Phi(f\cdot N(g\cdot v))$. \Cref{thm:main} shows in particular that there is a subset $U\subseteq V_n$ such that $\Phi(f\cdot N(g\cdot v))$ is independent of $v\in U$. For $\tau=(13)$, $f=1+\tau$, $g=1-\tau$ and $v=(a,b,c)$, this interpretation coincides of course with the $1089$-trick (see \Cref{ex:neg}), with the result being constant on $U=\{(a,b,c)\in V_3 : a>c\}$. 

We will furthermore generalize the distinction between the two cases $a>c$ and $a<c$ in the $1089$-trick: as we saw above, if $a>c$ then after action by $g=1-\tau$ we still obtain a three-digit number $abc-cba$, while if $a<c$ we do not. We will therefore call a nonzero number trick \emph{spectator friendly} if $\Phi(g\cdot U)\subseteq \N_n$, while in general we call it simply  a \emph{formal number trick} (see \Cref{def:trick}).

\subsection{Results}We are now ready to state the purpose of this note. We will in \Cref{sect:mainthm} show the following generalization of the 1089-trick:
\begin{itemize}
\item For any pair of zero divisors $f\circ g=0$ in $\Z[\Sigma_n]$, the value of $\Phi(f\cdot N(g\cdot v))$ for $v\in V_n$ depends only on the carries produced in the normalization process. In other words, the digits of the input numbers cancel out.
\item There is a partition (depending only on $g$) of $V_n$ into cells defined by linear inequalities, such that the normalization is constant on each cell of the partition. This, along with the previous point, constitute our main result, \Cref{thm:main}.
\item In addition to establishing the formal properties above, our aim is to provide several examples. We give a list of various classes of examples in \Cref{ex:null-rels} below, and in \Cref{table} we revisit some of those examples as well as others.
\end{itemize}

\begin{example}\label{ex:1089-cells}
According to the claims above, the 1089-number trick should give rise to a partition of $\N_3$ into cells on which the carrying produced in the normalization process is constant. We see this as follows: the carrying is constant on the subset $\{abc:a>c\}\subseteq\N_3$, on which the output of the computation is $1089$.

The carrying is also constant on the diagonal $\{abc:a=c\}$, but here the result of the computation is $0$. 

Finally, the carrying is constant on the remaining locus $\{abc:a<c\}$, on which our normalization process outputs $1010$ as we will see in \Cref{ex:neg}.

Hence the 1089-number trick gives rise to the partition of $\N_3$ into a disjoint union of cells,
$
\N_3=\{a>c\}\sqcup\{a=c\}\sqcup\{a<c\},
$
on which the output of the trick is respectively $1089$, $0$, and $1010$. In the terminology above, only the cell $\{a>c\}$ gives a spectator friendly number trick.
\end{example}

We summarize the above discussion by the following definition, which will be justified by \Cref{thm:main}:

\begin{definition}\label{def:trick}
A \emph{formal number trick} is an ordered triple $(f,g,U)$ consisting of zero divisors $f\circ g=0$ in $\Z[\Sigma_n]$, together with a subset $U\subseteq V_n$ such that $\Phi(f\cdot N(g\cdot v))$ is constant for all $v\in U$. 

A formal number trick $(f,g,U)$ is called a \emph{spectator friendly number trick}, or simply a \emph{number trick}, if $\Phi(g\cdot U)\subseteq \N_n$ and $\Phi(f\cdot N(g\cdot U))\neq0$.
\end{definition}

As an example, let $\tau=(13)\in\Z[\Sigma_3]$ and consider the triple $(1+\tau,1-\tau,U)$. If $U=\{a>c\}$ we obtain the usual $1089$-trick, which is spectator friendly in the sense of \Cref{def:trick}. If on the other hand $U$ is $\{a=b\}$ or $\{a<c\}$, we obtain a formal number trick which is not spectator friendly. See \Cref{ex:neg} for more details on how the $1089$-trick fits in this formalism.

Similarly, by switching the roles of $1+\tau$ and $1-\tau$, we obtain a formal number trick $(1-\tau,1+\tau,U)$, which is spectator friendly with output $99$ on the cell $U=\{a+c\le8,b\ge5\}$. See \Cref{ex:reverse-1089} for details on this ``reversed 1089-trick".

\begin{example}\label{ex:null-rels}
We can find several examples of formal number tricks by looking for zero divisors in $\Z[\Sigma_n]$. Below we list some different classes of examples; we invite the reader to find their own examples by using other zero divisors in $\Z[\Sigma_n]$. 
\begin{itemize}
\item[(a)] The classical 1089-trick uses the transposition $\tau=(13)$. Letting instead $\tau=(12)$, the zero divisors $(1+\tau)\circ (1-\tau)=0$ result in a spectator friendly number trick which outputs $990$ on the cell $\{a>b\}\subseteq\N_3$. Using instead $\tau=(23)$, the result is $99$ on the cell $\{b>c\}$. 

\item[(b)] Let $\rho=(123)$ denote the rotation in $\Sigma_3$, so that $\rho^3=1$. Then $(1+\rho+\rho^2)\circ (1-\rho)=0$, and this relation gives rise to the rotation trick we saw in the Introduction. One possible constraint is $a\ge b>c$, and this will give a spectator friendly  number trick whose output is 1998. This trick is spectator friendly on other cells as well; see \Cref{table} for more details.

\item[(c)] More generally, we can pick the rotation 
$\rho=(123\dots n)\in\Sigma_n$ together with the relation 
$
\p*{\sum_{i=0}^{n-1}\rho^i}\circ(1-\rho)=0,
$
which will define a formal number trick on $n$-digit numbers. One can show for instance that the output will always be a multiple of the $n$-th repunit $(10^n-1)/(10-1)=\underbrace{111\cdots 1}_{n}$.

\item[(d)] Behrends' \cite{Behrends} and Webster's \cite{1089} generalized 1089-trick on $n$-digit numbers is obtained from $(1+\sigma)\circ (1-\sigma)=0$, where $\sigma\in\Sigma_n$ reverses the digits. The papers \cite{Behrends,1089} contain results on the number of cells for these number tricks.
\item[(e)] For $H$ a nontrivial subgroup of $\Sigma_n$, let $N_H\defeq\sum_{h\in H}h\in\Z[\Sigma_n]$. If $h\in H$ then $h\circ N_H=N_H$, and so $N_H\circ(h-1)=0$. Hence the elements $N_H$ and $h-1$ of $\Z[\Sigma_n]$ define a formal number trick for any $h\in H\setminus\{1\}$.
\item[(f)] Generalizing the previous example, consider a nontrivial subgroup $H$ of $\Sigma_n$ together with a character $\chi\colon H\to\{\pm1\}$. Let 
$
N_{H,\chi}\defeq\sum_{h\in H}\chi(h)h\in\Z[\Sigma_n].
$
Then $N_{H,\chi}\circ (h-\chi(h))=0$ for any $h\in H$, and this defines a formal number trick if $h\neq1$.

As an example, take the Klein four-subgroup $H=\{1,x,y,z=xy\}\subseteq \Z[\Sigma_4]$, where $x=(12)(34)$, $y=(13)(24)$ and $z=(14)(23)$, along with the character $\chi(x)=\chi(y)=-1$, $\chi(z)=1$. For each $x$, $y$ and $z\in H$ we thus obtain a formal number trick. For instance, $N_{H,\chi}\circ(z-1)=(1-x-y+z)\circ(z-1)=0$, and on for instance the cell $\{a<d,b<c\}$ we obtain a spectator friendly number trick with output $1782$.
\end{itemize}
\end{example}

\section{Main result}\label{sect:mainthm}
In this section we prove our claims. We start by defining an action of $\Z[\Sigma_n]$ on length-$n$ digit vectors.

\subsubsection*{Symmetric group action.} Let
$
V_n=\{0,1,\dots,9\}^n\subseteq\Z^n
$
be the set of length-$n$ vectors $v=(v_1,\dots,v_n)$ where $v_i\in\{0,1,\dots,9\}$. The symmetric group $\Sigma_n$ acts on $V_n$ and $\Z^n$ by
\[
\sigma\cdot v=(v_{\sigma^{-1}(1)},\dots,v_{\sigma^{-1}(n)})
\]
for $v=(v_1,\dots,v_n)\in V_n$ or $\Z^n$, and $\sigma\in \Sigma_n$. In other words, we permute the coordinates of $v$ according to the permutation $\sigma$. We extend this action $\mathbb Z$-linearly to an action of the group ring $\mathbb Z[\Sigma_n]$ on $\mathbb Z^n$; thus for
$
g=\sum_{\sigma\in \Sigma_n} a_\sigma\,\sigma\in\mathbb Z[\Sigma_n]
$
and $v\in\mathbb Z^n$ we write $g\cdot v=\sum_\sigma a_\sigma(\sigma\cdot v)\in\mathbb Z^n$.

\subsubsection*{Keeping track of the carrying.}We now define suitable maps between $V_n$, $\Z^n$ and $\Z$ that allow us to keep track of the carrying and thus to formalize number tricks. First, define the evaluation homomorphism
\[
\Phi:\mathbb Z^n\to\mathbb Z,\qquad \Phi(x_1,\dots,x_n)=\sum_{i=1}^n x_i10^{\,n-i}.
\]
We now aim to define a ``normalization map" 
$
N\colon \Z^n\to V_n
$
which encodes the carrying operation. 

\begin{definition}\label{def:N}
For $u=(u_1,\dots,u_n)\in \Z^n$, define the \emph{normalized vector}
\[
N(u)\defeq(d_1,\dots, d_n)\in V_n=\{0,1,\dots,9\}^n,
\]
where the coordinates $d_i$ of $N(u)$ are defined recursively as follows. For $i=n,n-1,\dots,1$, set:
\begin{align*}
&t_n\defeq u_n,\quad c_n\defeq\left\lfloor\dfrac{u_n}{10}\right\rfloor, \quad d_n\defeq t_n-10c_n\\[10pt]
&t_i\defeq u_i+c_{i+1},\quad c_i\defeq\left\lfloor\dfrac{t_i}{10}\right\rfloor,\quad d_i\defeq t_i-10c_i.
\end{align*}
In other words, the $d_i$'s are obtained by the usual carrying procedure from right to left.

Finally, for $u\in\mathbb Z^n$ define the \emph{carry vector}
\[
\bfc(u)\defeq N(u)-u\in\mathbb Z^n.
\]
\end{definition}

\begin{remark}We note the following:
\begin{itemize}
\item The map $\Phi\colon\Z^n\to \Z$ is $\Z$-linear, while the map $N\colon\Z^n\to V_n$ not. Note also that the restriction $\Phi|_{V_n}\colon V_n\to\{0,1,\dots,10^n-1\}$ is bijective. See \Cref{lem:normalization} below for more details on the relationship between the maps $\Phi$ and $N$.
\item There are two ways to record the carrying involved in the normalization process. One is via the carry vector $\bfc(u)$ defined above as $N(u)-u$. Another way is to collect the $c_i$'s occurring in the normalization algorithm of \Cref{def:N} into a vector $(c_1,\dots,c_n)$. These two vectors are in general different: for instance, we will see in \Cref{ex:neg} that for the $1089$-trick, the carry vector $\bfc(g\cdot v)$ is $(-1,9,10)$, while the vector $(c_1,c_2,c_3)$ is $(0,-1,-1)$. In the proof of \Cref{thm:main} it will however be most convenient to work with $\bfc(u)=N(u)-u$. We will therefore refer to $\bfc(u)$ as the carry vector, and refer to the $c_i$'s as the \emph{algorithmic carries} produced by the normalization algorithm.
\item A formal number trick $(f,g,U)$ with $\Phi(f\cdot N(g\cdot U))\ne0$ is spectator friendly if and only if the final algorithmic carry $c_1$ is zero.
\end{itemize}
\end{remark}

\begin{example}\label{ex:neg}
Let us see how the classical 1089-trick fits in the formalism of \Cref{def:N}. Let $v=(a,b,c)\in V_3=\{0,1,\dots,9\}^3$ with $a>c$, $\tau=(13)\in\Sigma_3$, $f=1+\tau$, and $g=1-\tau$. Then, with the notation above, the 1089-number trick means the computation of the number
$
\Phi(f\cdot N(g\cdot v)).
$
Thus we must first compute the normalization $N(g\cdot v)=(d_1,d_2,d_3)$ of $g\cdot v=(a-c,0,c-a)$, which simply means writing down the result after carrying. Indeed, we first find that $c_3=\floor{(c-a)/10}=-1$ (since $0\le c<a\le 9$) and hence $d_3=(c-a)-10c_3=10+c-a$. Similarly we find $d_2=9$ and $d_1=a-c-1$, so $N(g\cdot v)=(a-c-1,9,10+c-a)$. Hence
\begin{align*}
\Phi(f\cdot N(g\cdot v))&=\Phi\big((a-c-1,9,10+c-a)+(10+c-a,9,a-c-1)\big)\\
&=\Phi(9,18,9)=1089.
\end{align*}
We note also that in this example, the carry vector $\bfc(g\cdot v)$ of $g\cdot v$ is 
\[
\bfc(g\cdot v)=N(g\cdot v)-g\cdot v=(-1,9,10),
\]
while the algorithmic carries are given by
\[
(c_1,c_2,c_3)=(0,-1,-1).
\]
The final algorithmic carry $c_1$ being zero is one way of saying that the trick $(1+\tau,1-\tau,\{a>c\})$ is spectator friendly.

We can similarly find the output when $a<c$: indeed, we run the same algorithm as above and find $d_3=c-a$, $d_2=0$, $d_1=10+a-c$, and finally $\Phi(f\cdot N(g\cdot v))=\Phi(10,0,10)=1010$.
In this case, the carry vector is $\bfc(g\cdot v)=(10,0,0)$ and the algorithmic carries are given by $(c_1,c_2,c_3)=(-1,0,0)$. Since $c_1$ is here nonzero, the formal number trick $(1+\tau,1-\tau,\{a<c\})$ is not spectator friendly. 
\end{example}

\begin{example}\label{ex:reverse-1089}
Let $\tau=(13)$, so that the 1089-trick is given by $f=1+\tau$ and $g=1-\tau$. If we instead let $f=1-\tau$ and $g=1+\tau$, we obtain a formal number trick which is spectator friendly only on the cell $\{a+c\le8,b\ge5\}$. Indeed, for $v=(a,b,c)\in V_3$, we have
\[
g\cdot v =(a,b,c)+(c,b,a)= (a+c,\,2b,\,a+c).
\]
We need at least one carry for the result to be nonzero after applying $f=1-\tau$, and this carry has to come from the middle term $2b$ (since if $a+c\ge10$ then $\Phi(a+c,2b,a+c)$ lands outside $\N_3$). This gives the constraints $a+c\le8$, $2b\ge10$. Under these constraints we find
$N(g\cdot v)=(a+c+1,2b-10,a+c)$ and $\Phi(f\cdot N(g\cdot v))=\Phi(1,0,-1)=99$.
\end{example}

\begin{lemma}\label{lem:normalization}
The following diagram commutes in the category of sets: 
\[\begin{tikzcd}
\Z^n\ar{r}{N}\ar{d}[swap]{\Phi} & V_n\ar{d}{\Phi|_{V_n}}\\
\Z\arrow{r}[swap]{\bmod{10^n}} & \Z/10^n
\end{tikzcd}\]
In other words, for any $u=(u_1,\dots,u_n)\in \mathbb Z^n$ we have
$
N(u) \equiv (\Phi|_{V_n})^{-1}\circ\Phi(u) \bmod{10^n}.
$
In particular, $\Phi|_{V_n}(N(u))$ is the unique representative of the congruence class
$\Phi(u) \bmod 10^n$ in the set $\{0,1,\dots,10^n-1\}$.
\end{lemma}

\begin{proof}
By definition,
the normalization $N(u)=(d_1,\dots,d_n)$ is obtained by
$
d_i= u_i + c_{i+1} - 10 c_i
$
with $c_{n+1}=0$. Thus
$
d_i - u_i=c_{i+1} - 10 c_i,
$
so that
\[
N(u)-u = (d_1-u_1,\dots,d_n-u_n) = (c_2-10c_1, c_3-10c_2,\dots,-10c_n).
\]
We now apply $\Phi$ to obtain
\[
\Phi(N(u))-\Phi(u) 
= \sum_{i=1}^n (c_{i+1}-10c_i)\cdot 10^{n-i}
= \sum_{i=1}^nc_{i+1}\cdot10^{n-i}-\sum_{i=1}^nc_i\cdot10^{n-i+1}
\]
Here all terms except $-c_1\cdot10^n$ cancel, yielding
$
\Phi(N(u))-\Phi(u) = -c_1\cdot 10^n.
$
In other words,
\[
\Phi(N(u))\equiv \Phi(u)\bmod{10^n}.
\]
Finally, since $N(u)\in\{0,1,\dots,9\}^n$, it follows that $\Phi(N(u))$ lies between $0$ and $10^n-1$. This proves uniqueness of the representative.
\end{proof}

We are now ready to state our main result:

\begin{theorem}\label{thm:main}
Let $g\in\mathbb Z[\Sigma_n]$. For any $v\in\Z^n$, write $g\cdot v=(u_1(v),\dots,u_n(v))$. Let also $c_1,\dots,c_n$ denote the algorithmic carries that occur the normalization algorithm for $g\cdot v$, so that $\bfc(g\cdot v)=(e_1,\dots,e_n)$ where $e_i=c_{i+1}-10c_i$.
\begin{enumerate}
\item (Partition into carry cells)
There is a finite subset $\mathcal C\subseteq\mathbb Z^n$ such that $V_n$ can be written as a disjoint union
\[
V_n=\bigsqcup_{\bfe\in\mathcal C} U_{\bfe},\quad
U_{\bfe}\defeq\{v\in V_n:\bfc(g\cdot v)=\bfe\}\quad\text{where }\bfe=(e_1,\dots,e_n)\in\calC,
\]
where each cell $U_{\bfe}$ is the set of solutions in $V_n$ to the system of linear inequalities
\begin{align*}
10c_n \le u_n(v)&\le 10c_n+9\\[10pt]
10 c_i \le u_i(v)+c_{i+1}&\le 10c_i+9\quad(i=n-1,\dots,1)
\end{align*}
for integers $c_1,\dots,c_n$ (necessarily the algorithmic carries for $g\cdot v$)
satisfying $e_i=c_{i+1}-10c_i$.

\item (Constancy on cells)
On every nonempty cell $U_{\bfe}$ of the partition, the carry vector $\bfc(g\cdot v)$ is constant, equal to $\bfe$, and therefore
\[
N(g\cdot v) \;=\; g\cdot v + \bfe
\qquad\text{for all } v\in U_{\bfe}.
\]

\item (Formal number trick)
If $f\in\mathbb Z[\Sigma_n]$ satisfies $f\circ g=0$ in $\mathbb Z[\Sigma_n]$, then for every $v\in V_n$,
\[
\Phi\big(f \cdot N(g\cdot v)\big)
=\Phi\big(f\cdot \bfc(g\cdot v)\big).
\]
In other words, the output is constant on each cell $U_{\bfe}$ and equals
$
\Phi(f\cdot \bfe)
$ there.
\end{enumerate}
\end{theorem}

\begin{proof}
Each coordinate $u_i(v)$ of $g\cdot v$ is an integer linear form in the digits of $v$, so as $v$ ranges over $V_n$ it takes finitely many values. If $\bfe=(e_1,\dots,e_n)\in\mathbb Z^n$, then $\bfc(g\cdot v)=\bfe$ if and only if there exist integers $c_1,\dots,c_n$ such that $e_n=-10c_n$ and $e_i=c_{i+1}-10c_i$ for $i=n-1,\dots,1$, and which furthermore satisfies
\[
10c_n \le u_n(v)\le 10c_n+9,\quad
10c_i \le u_i(v)+c_{i+1}\le 10c_i+9\quad (i=n-1,\dots,1).
\] 
Thus $U_{\bfe}$ is (the integer points of) a polytope intersected with the box $V_n$; only finitely many $\bfe$ can occur, yielding a finite partition of $V_n$. On $U_{\bfe}$ the algorithmic carries are by definition constant, hence $N(g\cdot v)=g\cdot v+\bfe$ there. If $f\circ g=0$, then by linearity of $\Phi$ we have
\[
0=\Phi\big(f\cdot(g\cdot v)\big)
=\Phi\big(f\cdot(N(g\cdot v)-\bfc(g\cdot v))\big)
\]
for any $v\in\Z^n$, yielding $\Phi(f\cdot N(g\cdot v))=\Phi(f\cdot \bfc(g\cdot v))$ which is constant on each $U_{\bfe}$.
\end{proof}

\begin{remark}\begin{itemize}
\item \Cref{thm:main} shows that the number of cells in the partition $V_n=\bigsqcup_\bfe U_\bfe$ is determined solely by the choice of $g$.
The role of $f$ (in a null relation $f\circ g=0$) is to determine which constant
value is assigned to each cell.

\item Dropping the assumption that $f$ and $g$ are zero divisors may lead to pathological partitions of $V_n$ into singleton cells. For instance, if $k\ge1$ and $g=10^k\id\in\Z[\Sigma_n]$, then the carry map $v\mapsto\bfc\big(g\cdot v)=N(g\cdot v)-g\cdot v\colon V_n \to \mathbb Z^n$ is injective, implying that each $U_{\bfe}$ in the corresponding partition is a singleton set.
\end{itemize}
\end{remark}

\section{Examples}\label{table}
With the technical setup of the previous section at hand, we now turn back to examples.

\subsection*{Transposition trick with $\tau=(12)$: $f=1+\tau$, $g=1-\tau$}
Here $g\cdot(a,b,c)=(a-b,b-a,0)$, and the partition of $\N_3$ is given by the sign of $a-b$. We record the cells, output, and the algorithmic carries below. We also mark if the given triple $(f,g,U)$ is spectator friendly.

\[
\begin{array}{c|c|c|c|c}
\text{\vtop{\hbox{\strut Cell condition on}\hbox{\strut \qquad $(a,b,c)$}}} & \bfe & \text{\vtop{\hbox{\strut Algorithmic carries }\hbox{\strut \qquad $(c_1,c_2,c_3)$}}} & \Phi\big(f\cdot N(g\cdot v)\big) 
 & \text{\vtop{\hbox{\strut Spectator}\hbox{\strut friendly?}}}\\
\midrule
a>b & (-1,\,10,\,0) & (0,\,-1,\,0) & 990   &  \checkmark\\
a=b & (0,\,0,\,0)   & (0,\,0,\,0)  & 0     &  \\
a<b & (10,\,0,\,0)  & (-1,\,0,\,0)  & 1100  & 
\end{array}
\]

\subsection*{Transposition trick with $\tau=(23)$: $f=1+\tau$, $g=1-\tau$}
Here $g\cdot(a,b,c)=(0,\ b-c,\ c-b)$, and the partition is by the sign of $b-c$.

\[
\begin{array}{c|c|c|c|c}
\text{\vtop{\hbox{\strut Cell condition on}\hbox{\strut \qquad $(a,b,c)$}}}  & \bfe & \text{\vtop{\hbox{\strut Algorithmic carries }\hbox{\strut \qquad $(c_1,c_2,c_3)$}}} & \Phi\big(f\cdot N(g\cdot v)\big) 
 & \text{\vtop{\hbox{\strut Spectator}\hbox{\strut friendly?}}}\\
\midrule
b>c  & (0,\,-1,\,10) & (0,\,0,\,-1) & 99    & \checkmark \\
b=c  & (0,\,0,\,0)   & (0,\,0,\,0)  & 0     &  \\
b<c  & (9,\,10,\,0)  & (-1,\,-1,\,0)  & 1910  & 
\end{array}
\]

\subsection*{Rotation trick with $\rho=(123)$: $f=1+\rho+\rho^2$, $g=1-\rho$}
Here $g\cdot(a,b,c)=(a-c,\ b-a,\ c-b)$.

\[
\begin{array}{c|c|c|c|c}
\text{\vtop{\hbox{\strut Cell condition on}\hbox{\strut \qquad $(a,b,c)$}}} & \bfe & \text{\vtop{\hbox{\strut Algorithmic carries }\hbox{\strut \qquad $(c_1,c_2,c_3)$}}} & \Phi\big(f\cdot N(g\cdot v)\big) 
 & \text{\vtop{\hbox{\strut Spectator}\hbox{\strut friendly?}}}\\
\midrule
a=b=c                          & (0,\,0,\,0)   & (0,\,0,\,0)  & 0    &  \\
c\le a< b                      & (0,\,-1,\,10) & (0,\,0,\,-1)  & 999  & \checkmark \\
b\le c< a                      & (-1,\,10,\,0) & (0,\,-1,\,0) & 999  &  \checkmark\\
c< b\le a                      & (-1,\,9,\,10) & (0,\,-1,\,-1) & 1998 &  \checkmark\\
a\le b\le c\text{ and }a<c     & (10,\,0,\,0)  & (-1,\,0,\,0)  & 1110 &  \\
a< c< b                        & (10,\,-1,\,10)& (-1,\,0,\,-1) & 2109 &  \\
b< a\le c                      & (9,\,10,\,0)  & (-1,\,-1,\,0)  & 2109 & 
\end{array}
\]

\begin{remark}\label{remark:null}
One may readily extend \Cref{thm:main} to null relations $\sum_if_i\circ g_i=0$ in $\Z[\Sigma_n]$ where not necessarily each $f_i\circ g_i$ equals $0$. The associated formal number trick is computed as $\sum_i\Phi(f_i\cdot N(g_i\cdot v))=\sum_i\Phi(f_i\cdot \bfc(g_i\cdot v))$. In this way, any null relation in $\Z[\Sigma_n]$ defines a formal number trick.
\end{remark}

\begin{example}
We round off with a didactic example illustrating that in principle, school students can discover their own number tricks by playing with relations between symmetries. 

Suppose the students are learning about symmetries, and are asked to find the symmetries of an equilateral triangle. They find that the nontrivial symmetries consist (using the notation of \Cref{figure}) of the rotations $\rho=(123)$ and $\rho^2$, along with the reflections $\mu_1=(23)$, $\mu_2=(13)$ and $\mu_3=(12)$. 
The students learn further that we can follow up one symmetry by another, and that this yields another symmetry. For instance, the reflection $\mu_3$ followed by the rotation $\rho$ results in the reflection $\mu_2$. The ambitious student may even derive the entire multiplication table for $\Sigma_3$.
\begin{figure}
\begin{tikzpicture}[>=stealth, every node/.style={font=\small}]
\def\s{3}
\coordinate (A) at (0,0);
\coordinate (B) at (\s,0);
\coordinate (C) at ({\s/2},{0.8660254*\s}); 

\coordinate (Mab) at ($(A)!0.5!(B)$);
\coordinate (Mbc) at ($(B)!0.5!(C)$);
\coordinate (Mca) at ($(C)!0.5!(A)$);

\coordinate (O)   at ($(C)!2/3!(Mab)$);

\draw[thick] (A)--(B)--(C)--cycle;
\node[anchor=north east] at ($(A)+(-2pt,-2pt)$) {$a$};
\node[anchor=north west] at ($(B)+(2pt,-2pt)$) {$b$};
\node[anchor=south]      at ($(C)+(0,2pt)$)     {$c$};

\pgfmathsetmacro{\r}{0.577350269*\s} 
\def\delta{5}

\draw[->, thick] (O) ++(-150+\delta:\r) arc (-150+\delta:-30-\delta:\r);

\draw[->, thick] (O) ++(-30+\delta:\r)  arc (-30+\delta:90-\delta:\r);

\draw[->, thick] (O) ++(90+\delta:\r)   arc (90+\delta:210-\delta:\r);

\node at ($(O)+(-90:0.8*\r)$) {$\rho$};

\draw[dashed] (A)--(Mbc);
\draw[dashed] (B)--(Mca);
\draw[dashed] (C)--(Mab);

\node at (0.7,0.65) {$\mu_1$};
\node at (2.3,0.65) {$\mu_2$};
\node at (1.27,1.6) {$\mu_3$};

\begin{scope}[xshift=8.1cm, yshift=0.2cm]
  \matrix (T) [matrix of nodes,
               nodes in empty cells,
               nodes={minimum width=8mm, minimum height=6mm, inner sep=1pt, anchor=center},
               row sep=-\pgflinewidth, column sep=-\pgflinewidth,
               row 1/.style={nodes={fill=gray!15}},
               column 1/.style={nodes={fill=gray!15}}] {
      {}   & $1$ & $\rho$ & $\rho^2$ & $\mu_1$ & $\mu_2$ & $\mu_3$ \\
      $1$   & $1$   & $\rho$   & $\rho^2$ & $\mu_1$ & $\mu_2$ & $\mu_3$ \\
      $\rho$& $\rho$& $\rho^2$ & $1$      & $\mu_3$ & $\mu_1$ & $\mu_2$ \\
      $\rho^2$ & $\rho^2$ & $1$ & $\rho$   & $\mu_2$ & $\mu_3$ & $\mu_1$ \\
      $\mu_1$ & $\mu_1$ & $\mu_2$ & $\mu_3$ & $1$    & $\rho$  & $\rho^2$ \\
      $\mu_2$ & $\mu_2$ & $\mu_3$ & $\mu_1$ & $\rho^2$ & $1$    & $\rho$ \\
      $\mu_3$ & $\mu_3$ & $\mu_1$ & $\mu_2$ & $\rho$ & $\rho^2$ & $1$ \\
  };
\end{scope}
\end{tikzpicture}
\caption{The symmetries of a triangle can be used to generate number tricks.}
\label{figure}
\end{figure}

This is the point where number tricks enter the picture: it is precisely these kinds of relations between symmetries that give rise to number tricks. Indeed, knowing for instance that $\rho\circ\mu_3=\mu_2$, we can use this equality to derive a null relation as in \Cref{remark:null} and hence a (formal) number trick.
How to look for such a null relation? One possible recipe is:
\begin{itemize}
\item Start with a relation between symmetries that the students have discovered, say $\rho\circ\mu_3=\mu_2$.
\item Use this relation to create a two term null relation $f_1\circ g_1+f_2\circ g_2=0$ in $\Z[\Sigma_3]$, which will define a (formal) number trick as in \Cref{remark:null}. 
\end{itemize}
Keeping the spirit of the $1089$-trick, let us have the trick starting with the term $1-\mu_2$.
In the second term we can then make use of the equality $\rho\circ\mu_3=\mu_2$ to kill off $\mu_2$ from the first term. Then we only lack something to kill off the $1$, which could for instance be done by using that $\rho^3=1$. Now we have suitable candidates for our second term: we let $f_2=\rho$ and $g_2=\mu_3-\rho^2$. Our formal number trick is then given by the null relation
\[
(1-\mu_2)+\rho\circ(\mu_3-\rho^2)=0.
\]
On the cell $\{a>c\}$ this is a spectator friendly number trick with output $999$. Performing this number trick goes as follows:
\begin{itemize}
\item[Step 1:] Choose a three-digit number $abc$ (with $a>c$), say 321, and subtract its reverse: $321-123=198$. 
\item[Step 2:] Swap the first and second digits of the chosen number, and subtract the doubly rotated number: $231-213=018$. Finally, rotate this result and add it to the result from the first step: $801+198=999$. 
\end{itemize}

Ambitious students may venture further into other dihedral groups. We can for instance consider symmetries of the square, which we can use to produce formal number tricks on four-digit numbers. As an example, take $\tau=(14)$, $\mu=(23)$ and $\sigma=\tau\circ\mu\in D_4$. Then the null relation
\[
(1-\sigma)+\tau\circ(\mu-\tau)=0\in\Z[\Sigma_4]
\]
gives a formal number trick. It is spectator friendly on for instance the cell $\{a>d,b>c\}$, on which the output is $10989$.

In this way, the students may compute relations between symmetries and use them to create their own number tricks, secretly learning group theory in the process.
\end{example}

\printbibliography
\end{document}